\documentclass{article}
\usepackage{ifthen}
\def\style{}
\ifthenelse{\equal{\style}{preprint}}{
	\usepackage{dmo_preprint}}{}
\usepackage{amsmath,amsthm,amssymb,amsfonts,epsfig}

\newtheorem{thm}{Theorem}
\newtheorem{cor}[thm]{Corollary}
\newtheorem{prop}[thm]{Lemma}

\newtheorem{definition}[thm]{Definition}
\newtheorem{ex}[thm]{Example}

\title{Kirchberger's Theorem for Complexes of Oriented Matroids}
\author{Winfried Hochst\"{a}ttler$^1$ \qquad \qquad  Sophia Keip$^1$ \qquad \qquad Kolja Knauer$^2$\\
\normalsize $^1$FernUniversit\"{a}t in Hagen, Fakult\"{a}t f\"{u}r 
Mathematik und Informatik\\
\normalsize 58084 Hagen, Germany\\
\normalsize $^2$Departament de Matem\`atiques i Inform\`atica,
Universitat de Barcelona, Spain\\
\normalsize \texttt{\{winfried.hochstaettler, sophia.keip\}}@fernuni-hagen.de\\
\normalsize \texttt{kolja.knauer}@ub.edu}

\begin{document}
 \ifthenelse{\equal{\style}{preprint}}
{
 \DMOmathsubject{52C40, 05B35, 52A35}
   \DMOkeywords{Kirchberger's Theorem, oriented matroids, COMs} 
\DMOtitle{062.19}{Kirchberger's Theorem for Complexes of Oriented Matroids}{Winfried Hochst\"{a}ttler, Sophia Keip, Kolja Knauer}{\{winfried.hochstaettler, sophia.keip\}@fernuni-hagen.de,\\ \texttt{kolja.knauer}@lis-lab.fr}
}

\maketitle

\begin{abstract}
The separation theorem of Kirchberger can be proven using a combination of Farkas' Lemma and Carath\'{e}odory's Theorem. Since those theorems are at the heart of oriented matroids, we are interested in a generalization of Kirchberger's Theorem to them. This has already been done for rank 3 oriented matroids. Here we prove it for complexes of oriented matroids, which are a generalization of oriented matroids.


\end{abstract}

\section{Introduction}

In order to introduce Kirchberger's Theorem we use a picture from \cite{valentine}. Imagine we have black and white sheep in a meadow and we want to decide whether they can be separated by a straight fence. Kirchberger gives an answer to this question.

\begin{thm}[Kirchberger's Theorem]
Let $V$ and $W$ be finite subsets of $\mathbb{R}^n$. If every set $C \subseteq V \cup W$ of $n+2$ or fewer points can be strictly separated into the sets $V \cap C$ and $W \cap C$, then $V$ can be strictly separated from $W$, i.e.\,one can find $a \in \mathbb{R}^n$ and $\alpha \in \mathbb{R}$ such that $a^Tv - \alpha < 0$ for all $v \in V$ and $a^Tw - \alpha > 0$ for all $w \in W$
\end{thm}

For our example this means if every set of four sheep can be separated by a straight fence, all sheep can be separated, see figure \ref{fig:1}. 

\begin{figure}\label{fig:1}
\center
\includegraphics[scale=0.25]{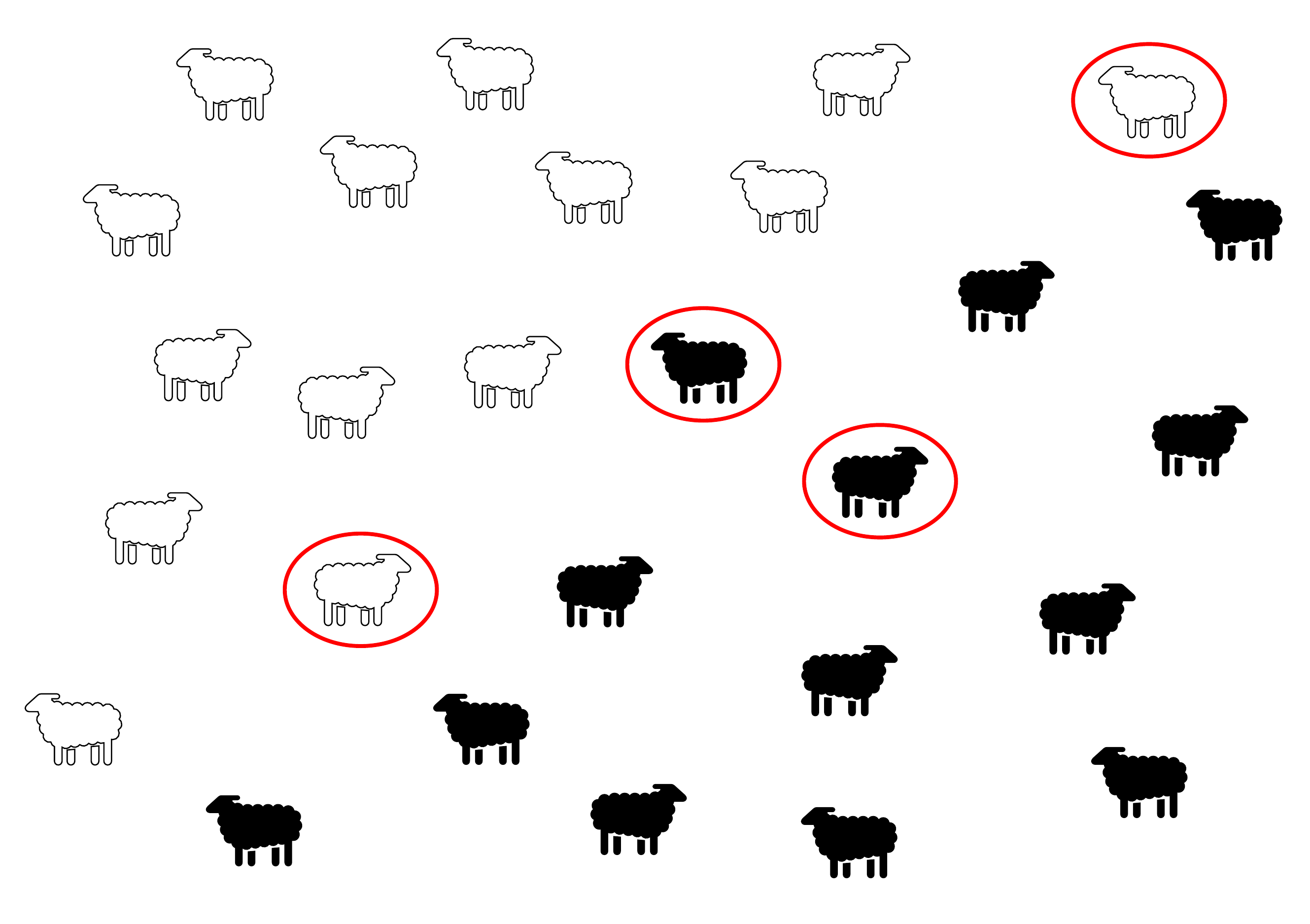}
\caption{Black and white sheep in the plane which obviously can not be separated by a straight fence. In this case we find a set of four sheep, where separation already fails.}
\end{figure}

The original proof of Kirchberger in 1902 is really long and hard to understand \cite{kirchberger}. Nowadays easier proofs are known. One possibility is to prove it using Helly's Theorem like in \cite{barvinok} or \cite{schoenberg}. There is also a simpler proof which is basically a combination of Carath\'{e}odory's Theorem and Farkas' Lemma which can be found in \cite{webster}. Because those two theorems are crucial for \emph{oriented matroids (OMs)}, it is natural to generalize Kirchberger's Theorem to them as well. This has been already done for pseudoline arrangements, i.e.\,OMs of rank $3$ \cite{bergold}, \cite{cordovil}. We will prove it for \emph{complexes of oriented matroids  (COMs)}. COMs have been introduced in~\cite{bandelt} as a common generalization of oriented matroids, affine oriented matroids, and lopsided sets. Alternatively, they have been called conditional oriented matroids.

\section{Basics about COMs}
Before we define COMs and some of their properties, we need the following definitions

\begin{definition}
Let $\mathcal{L} \subseteq \{0,+,-\}^E$ be a set of sign vectors on a finite ground set $E$. The \emph{composition} of two sign vectors $X$ and $Y$ is defined as

$$(X \circ Y)_e = \begin{cases}
                        X_e & \text{ if } X_e \neq 0,\\
                        Y_e & \text{ if } X_e = 0,\\
                       \end{cases} 
                    \forall e\in E.$$

\noindent The \emph{separator} of $X$ and $Y$ is defined as
\begin{align*}
S(X,Y) = \{e \in E: X_e = -Y_e \neq 0\}.
\end{align*}
The \emph{support} of X is defined as
\begin{align*}
\underline{X} = \{e \in E: X_e \neq  0\}.
\end{align*}
\end{definition}
Let us introduce three axioms for systems of sign vectors.
\begin{description}
 
\item[(FS)]\emph{Face Symmetry} 
\begin{align*}
\forall X,Y \in \mathcal{L}: X \circ (-Y) \in \mathcal{L}
\end{align*}
\item[(SE)]\emph{Strong Elimination} 
\begin{align*}
&\forall X,Y \in \mathcal{L}\, \forall e \in S(X,Y)\, \exists Z \in \mathcal{L}: \\
&Z_e=0 \text{ and }\forall f \in E \setminus S(X,Y): Z_f = (X \circ Y)_f.
\end{align*} 

\item[(C)]\emph{Composition} 
\begin{align*}
\forall X,Y \in \mathcal{L}: X \circ Y \in \mathcal{L}.
\end{align*}
\end{description}

Now we are in the position to define the term COM.

\begin{definition}[Complex of Oriented Matroids (COM)] Let $E$ be a finite set and $\mathcal{L} \subseteq \{0,+,-\}^E$. The pair $\mathcal{M}=(E,\mathcal{L})$ is called a COM, if $\mathcal{L}$ satisfies (FS) and (SE). The elements of $\mathcal{L}$ are called \emph{covectors}.
\end{definition}

Note that (FS) implies (C). Indeed, by (FS) we first get $X\circ -Y\in\mathcal{L}$ and then $X\circ Y= (X\circ -X)\circ Y= X\circ -(X\circ -Y) \in \mathcal{L}$ for all $X,Y\in \mathcal{L}$. 
This observation allows to define OMs as special COMs: 
\begin{definition}[Oriented Matroid (OM)] Let $E$ be a finite set and $\mathcal{L} \subseteq \{0,+,-\}^E$. The pair $\mathcal{M}=(E,\mathcal{L})$ is called a OM, if $\mathcal{L}$ satisfies (FS), (SE) and the all zeros vector $\mathbf{0}\in\mathcal{L}$.
\end{definition}

Let $\mathcal{M}=(E,\mathcal{L})$ be a COM. In the following we assume that $\mathcal{M}$ is \emph{simple}, i.e. $\forall e \in E: \{X_e|X \in \mathcal{L}\} = \{+,-,0\}$ and $\forall e \neq f \in E: \{X_eX_f | X \in \mathcal{L}\} = \{+,-,0\}$. In this setting the sign-vectors in $\mathcal{L}$ of full support are called \emph{topes} and $\mathcal{T}$ is the set of all {topes} of $\mathcal{M}$. A COM $\mathcal{M}$ is an \emph{oriented matroid (OM)} \cite{bjorner}, if $\mathbf{0} \in \mathcal{L}$.

The \emph{restriction} of a sign-vector $X\in\{0, \pm\}^E$ to $E \backslash F$, $F \subseteq E$, denoted by $X\backslash F \in \{0,+,-\}^{E\backslash F}$, is defined by $(X \backslash F)_e = X_e$ for all $e \in E \backslash F$. Given a system of sign vectors $\mathcal{M}=(E,\mathcal{L})$ and $F\subseteq E$, the \emph{contraction} of $F$ is the system of sign vectors $\mathcal{M}/F=(E\backslash F, \mathcal{L}/F)$, where $\mathcal{L}/F = \{X\backslash F: X \in \mathcal{L} \text{ and } \underline{X} \cap F = \emptyset\}$. It has been shown in~\cite{bandelt} and we will implicitly make use of it that the class of COMs is closed under contractions.  Let us look at an example of a {COM}:\\
\begin{ex}\label{ex:com} Let $E = \{v_1,\dots,v_m\} \subset \mathbb{R}^n$. We look at the following functions
\begin{align*}
(a,\alpha):\, E &\rightarrow \{+,-,0\}\\ 
v_i &\rightarrow \text{sign}(a^Tv_i-\alpha),
\end{align*}
where $a \in \mathbb{R}^n$, $\alpha \in \mathbb{R}$ and $i = 1\dots n$. We claim that the collection of those functions induce a COM with ground set $E$ and covectors $(\text{sign}(a^Tv_1-\alpha),\dots,\text{sign}(a^Tv_n-\alpha))$. Let $X$ be induced by $(a,\alpha)$ and $Y$ be induced by $(b,\beta)$. We set 
\begin{align*}
\epsilon = \min\Bigg\{\frac{|a^Tv_i-\alpha|}{|b^T v_i-\beta|}:\,|a^T v_i-\alpha|\cdot | b^T v_i -\beta| \neq 0\Bigg\}.
\end{align*}
Now the sign vector $X \circ -Y$ can be induced by 
\begin{align*}
(c,\gamma) = (a,\alpha) - \frac{\epsilon}{2}(b,\beta).
\end{align*}
One can see this by looking at
\begin{align*}
(X \circ -Y)_i = \text{\emph{sign}}(c^T v_i - \gamma) = \text{\emph{sign}}((a^T v_i - \alpha) - \frac{\epsilon}{2} (b^T v_i - \beta)).
\end{align*}
This equals $X_i=\text{\emph{sign}}(a^T v_i - \alpha)$, if $X_i \neq 0$ and $-Y_i = -\text{\emph{sign}}(b^T v_i - \beta)$, if $X_i = 0$. Since $(c,\gamma)$ is in our collection of functions, we see that  $X \circ -Y$ is in the COM, so
 \emph{face symmetry} is fulfilled. Let us look at \emph{strong elimination}. Let $e \in S(X,Y)$ and w.l.o.g. $(a^Tv_e-\alpha) < 0$ and $(b^Tv_e-\alpha) > 0$. If we look at the vector
\begin{align*}
Z_i = \text{\emph{sign}}((b^Tv_e-\alpha)(a^Tv_i-\alpha)-(a^Tv_e-\alpha)(b^Tv_i-\beta))
\end{align*}
we see that $Z_e = 0$ and $Z_f = (X \circ Y)_f$ for $f \in E\backslash S(X,Y)$. Furthermore the function that induces $Z$ is in our collection, so \emph{strong elimination} is fulfilled as well and we have a COM. Note that if we set $a=(0,\dots,0)$ and $\alpha = 0$ we get the sign vector $X = (0,\dots,0)$, so our COM is in particular an OM.
\end{ex}
Before we go on to Kirchberger's Theorem for COMs we need to define the rank of a COM.

\begin{definition}[Rank of a COM]\label{def:rank}  The rank $r(\mathcal{M})$ of a COM $\mathcal{M}=(E,\mathcal{L})$ is defined as
\begin{align*}
r(\mathcal{M}) = \max_{A \subseteq E} \big\{|A| \big| \mathcal{L}\backslash(E\backslash A) = \{ 0,+,-\}^{|A|} \big\}.
\end{align*}
\end{definition}

\section{Kirchberger's Theorem for COMs - Proof and Illustration}

So let $\mathcal{M}=(E,\mathcal{L})$ be a COM of rank $r$ on a ground set $E$ with $|E|=n$. We say two sets $V,W \subset E$ are \emph{separable} if there exists a covector $X = (X^+, X^-)$ such that $V \subseteq X^+$ and $W \subseteq X^-$. Our sheep correspond now to the elements of $E$ and as above we want to know if we can separate them. W.l.o.g. assume that we want to know if we can separate the first $k$ elements of $E$ from the last $n-k$ elements, i.e. we want to know if the vector
\begin{align}\label{Equ:wantedvector}
(\underbrace{+,+,\dots,+}_k,\underbrace{-,\dots,-,-}_{n-k})
\end{align}
is a \emph{tope} of $\mathcal{M}$. Our theorem will say that if for all $C \subseteq E$ with $|C| = r+1$ the sets $V \cap C$ and $W \cap C$ can be separated in $\mathcal{M}/(E\backslash C)$ (i.e.\,(\ref{Equ:wantedvector}) restricted to $C$ is a covector of $\mathcal{M}/(E\backslash C)$), then $V$ and $W$ can be separated in $\mathcal{M}$ (i.e.\,(\ref{Equ:wantedvector}) is a covector of $\mathcal{M}$). Let us demonstrate this in our example.

\begin{ex}[Example \ref{ex:com} continued]If now
\begin{align*}
X = (\underbrace{+,+,\dots,+}_k,\underbrace{-,\dots,-,-}_{n-k})
\end{align*} 
is a tope of the COM in Example \ref{ex:com} this means that there is an $(a,\alpha)$ such that $av-\alpha = 0$ separates $v_1,\dots,v_k$ and $v_{k+1},\dots,v_n$ strictly. So in this case Kirchberger's Theorem for COMs will say that whenever $r+1$ elements of $E$ can be separated strictly, then all of them can be separated strictly which is Kirchberger's Theorem in its original version. Let us look at the rank of our $COM$. We may assume (e.g.\,by induction over the dimension) that $v_1,\dots,v_n$ span $\mathbb{R}^n$ affinely. Therefore we will find vectors $v_{i_1}, \dots, v_{i_{n+1}}$ that span an n-simplex. It is easy to see (e.g.\,by induction) that one gets every possible sign vector within those simplex spanning elements by using a proper separating hyperplane. This shows by Definition \ref{def:rank} that the rank $r$ of our COM is at least $n+1$. We will show that the rank is exactly $n+1$. If we look at $n+2$ or more vectors,i.e.\,$V=v_{i_1}, \dots, v_{i_{n+2}}$, one would find by Radon's Theorem \cite{barvinok} a Radon Partition $(P_1,P_2)\subseteq V$, which is a partition where $conv(P_1)\cap conv(P_2) \neq \emptyset$. If we now look for the sign vector which has minus entries in $P_1$ and plus entries in $P_2$ we will see that this pattern can not be induced. Either some points of $P_1$ are in $conv(P_2)$, then it is obvious that they can not be separated from $P_2$ or there are two points of $P_1$ where the connecting line intersects $conv(P_2)$, so we also do not find an hyperplane which separates them from $P_2$. So the required sign vector can not be induced which shows that the rank is exactly $n+1$. By that we see that $r+1 = n+2$ which explains why we have $n+2$ in Kirchberger's original theorem and $r+1$ in the theorem for COMs.
\end{ex}

In order to simplify the proof we will formulate the theorem on a reorientation of $\mathcal{M}$ (i.e.\,$\mathcal{M}$ with some flipped signs, which does not affect the general structure),  where we do not look for the sign vector (\ref{Equ:wantedvector}) but for the all plus vector $R=\{+\}^n$. 

\begin{thm}[Kirchberger's Theorem for COMs]\label{thm:kirchbergerCOM}
Let $\mathcal{M}=(E,\mathcal{L})$ be a COM of rank $r$ and $|E|=n$. If for all $C \subseteq E$ with $|C| =r+1$ the sign vector $R \backslash (E \backslash C)$ is a tope of $\mathcal{M}/(E \backslash C)$, then $R$ is a tope of $\mathcal{M}$.
\end{thm}

We need the following lemma for our proof, which is a generalization of \cite[Lemma 4]{Hochstaettler}. The OM of the following example will play a major role in our proof.

\begin{ex}
Let us look at a special case of Example \ref{ex:com}. Take the points $\{e_i-e_{i+1}|1 \leq i \leq n-1\}\cup\{e_n-e_1\}$, where $e_i$ are the unit vectors. Any $n-1$ of them are linear independent but all $n$ of them are not. We call such structures a directed circuit and the corresponding COM (OM) $\mathcal{C}_n$. 
\end{ex}

\begin{prop}\label{lem:circuit}
Let $\mathcal{M}=(E,\mathcal{L})$ be a COM with tope set $\mathcal{T}$, such that   
  for all $f \in E$ there exists
  $T^f \in \mathcal{T}$ such that
  \[T^f_g=\left\{   \begin{array}[h]{rcl}
      + &\text{ if } & g \in E  \setminus\{f\}\\ 
      - &\text{ if } & g =f. 
    \end{array}\right.\]
    
  If $R\notin\mathcal{T}$, then $\mathcal{M} = \mathcal{C}_{|E|}$. 
\end{prop}

\begin{proof}
We will show by induction that all covectors which contain exactly one minus-entry and at least one plus-entry are in $\mathcal{L}$. Since then in particular all covectors which contain exactly one plus-entry and one minus-entry exist in $\mathcal{L}$, we get by (SE) that $\mathbf{O}\in\mathcal{L}$. Together, we can conclude that $\mathcal{M}=\mathcal{C}_n$, since we obtain all its covectors by composition of those vectors. Since $C_n$ is uniform no other oriented matroid can contain these covectors. 

So let $T^f \in \mathcal{T}$ for all $f \in E$ and $R\notin\mathcal{T}$. We will use induction over the number of zero-entries in the covectors, i.e.\,we want to show that for every $n = 0,\dots,|E|-2$ all sign-vectors with $n$ zero entries, one minus-entry and $|E|-(n+1)$ plus-entries are covectors of $\mathcal{M}$.  \\
$n=0$: By the existence of $T^f$ here is nothing to show. We fix $n>0$ and assume that all covectors with $n$ or less zero-entries, exactly one minus entry and at least one plus-entry exist in $\mathcal{L}$.\\
$n \rightarrow n+1 \leq |E|-2$: We now look for a covector with zero-entries in the i-th position, $i \in I \subset E$, $|I| = n+1$, a minus-entry in the j-th position, $j \notin I$ and $+$ everywhere else. We choose an $\hat{i} \in I$ and take two covectors with $0$ in $I \backslash \hat{i}$. One of them should have its $-$ in the $\hat{i}$-th position and the other one at the j-th position. W.l.o.g. those two covectors look like this:\\
\begin{align*}
&(0,\dots,0,\overbrace{-}^{\hat{i}},\;\,+\,\;,+,\dots,+)\\
&(\underbrace{0,\dots,0}_{I \backslash \hat{i}},\;\,+\,\;,\underbrace{-}_{j},+,\dots,+).
\end{align*}
They exist because $|I \backslash \hat{i}|=n$, so the induction hypotheses holds. If we now perform strong elimination with those two covectors we get (again w.l.o.g) the covector
\begin{align*}
X = (\underbrace{0,\dots,0}_{I \backslash \hat{i}},\underbrace{0}_{\hat{i}},\underbrace{*}_j,+,\dots,+).
\end{align*}
If $*$ was $+$, then $X \circ T^j = R$. Since $R\notin\mathcal{T}$ we have $*=-$ and have the covector we were looking for.
\end{proof}

We will now prove Theorem \ref{thm:kirchbergerCOM} by contraposition.
\begin{proof}
Suppose that $R$ does not exist in $\mathcal{L}$. Let now $D \subseteq E$ be of minimal cardinality such that $R \backslash (E \backslash D)$ does not exist in $\mathcal{M}/(E \backslash D)$. Since we choose $D$ to be minimal, we have that
\begin{align*}
\begin{array}{cccc}
(-, & +, & \dots, & +),\\
(+, & -, & \dots, & +),\\
\vdots  & &\ddots & \vdots\\
(+, & +, & \dots, & -)
\end{array} \in \mathcal{L}/(E \backslash D).
\end{align*}
Indeed, since $D$ is minimal for every $f\in E\setminus D$ there is a tope in $\mathcal{L}/D$ with $T^f_f=-$ being its only negative entry. By Lemma \ref{lem:circuit} we have $\mathcal{M}/(E \backslash D)=\mathcal{C}_{|D|}$, where $\mathcal{C}_{|D|}$ is the directed circuit of $|D|$ elements. Since $\mathcal{M}$ has rank $r$, the circuit can have at most $r+1$ elements, i.e.\,$|D| \leq r+1$. Therefore we can conclude that we will also find an $C$ with $|D|\leq |C| = r+1$ where $R \backslash (E \backslash C)$ will not exist in $\mathcal{M}/(E \backslash C)$, since it already did not exist for a smaller set. This finishes our contraposition.
\end{proof}

Since every OM is a COM, the statement for OMs is a direct corollary of Theorem \ref{thm:kirchbergerCOM}.

\begin{cor}[Kirchberger's Theorem for OMs]\label{cor:kirchbergerOM}
Let $\mathcal{O}=(E,\mathcal{L})$ be a OM of rank $r$ on $E$, $|E|=n$. If for all $C \subseteq E$ with $|C| = r+1$ the sign-vector $R \backslash (E \backslash C)$ exists in $\mathcal{O}/(E \backslash C)$, then $R$ exists in $\mathcal{O}$.
\end{cor}

\section{Discussion}
Reconsidering our result for oriented matroids we actually are a bit surprised that it previously had been considered only in the rank 3 case. Using some oriented matroid theory it is actually quite easy to see. By the topological representation theorem of Folkman and Lawrence every OM can be represented by an arrangement of oriented pseudospheres~\cite{folkman}. Here every circuit corresponds to a minimal system of closed hemispheres that cover the whole sphere~\cite{richter}. Let $H_{i}$ be the hemispheres belonging to an element of the support of a circuit $C$ and let $S^{r-1}$ be the ${(r-1)-}$dimensional sphere. We have that
\begin{align*}
\bigcup_{i \in \underline{C}} \overline{H_i} = S^{r-1}\\
S^{r-1} \backslash \bigcup_{i \in \underline{C}} \overline{H_i} = \emptyset\\
\bigcap_{i \in \underline{C}} S^{r-1} \backslash \overline{H_i} = \emptyset
\end{align*}
Since the set $S^{r-1} \backslash \overline{H_i}$, $i=1, \dots, |\underline{C}|$ corresponds to the open hemispheres corresponding to $-C$ we get
\begin{align*}
\bigcap_{i \in \underline{-C}} H_i = \emptyset.
\end{align*}
That means that the sign pattern associated with those hemispheres does not exist in any covector of the OM. Note that $-C$ is also a circuit of the OM.
Let us look at this the other way around: If we have a sign pattern, which is not a tope of our OM, clearly the intersection of the corresponding hemispheres is empty. Now we can delete hemispheres until we have a minimal system that covers the whole sphere, which gives us a circuit. So everytime a pattern is not a tope of the OM, we will find a circuit which can prove this. In the proof of Theorem \ref{thm:kirchbergerCOM} we show that this holds for COMs as well, despite the fact that no topological representation theorem is known for COMs.

\end{document}